\documentclass[12pt]{amsart}
\usepackage[active]{srcltx}
\usepackage{a4wide}
\usepackage{amsthm,amsfonts,amsmath,mathrsfs,amssymb}
\usepackage{dsfont}
\usepackage{fourier}
\def\be{\begin{equation}}
\def\ee{\end{equation}}

\newtheorem*{theorem*}{Theorem}
\newtheorem{theorem}{Theorem}

\newtheorem*{proposition*}{Proposition}
\newtheorem{lemma}{Lemma}

\theoremstyle{remark}
\newcommand{\nc}{\newcommand}

\newcommand{\D}{{\mathbb D}}

\nc{\supp}{\operatorname{supp}}

\nc{\dif}{\operatorname{d}} \nc{\im}{\operatorname{i}}
\nc{\Hi}{{\mathscr{H}}^\infty} \nc{\Ht}{{\mathscr{H}}^2}
\nc{\Hone}{{\mathscr{H}}^1} \nc{\ol}{\overline} \nc{\bz}{\mathbf{z}}
\nc{\bw}{\mathbf{w}} \nc{\eps}{\varepsilon}

\allowdisplaybreaks[3]

\begin{document}
\title{G\'{a}l-type GCD sums beyond the critical line}
\author{Andriy Bondarenko}
\address{Department of Mathematical Analysis\\ Taras Shevchenko National University of Kyiv\\
Volody- myrska 64\\ 01033 Kyiv\\ Ukraine}
\address{Department of Mathematical Sciences \\ Norwegian University of Science and Technology \\ NO-7491 Trondheim \\ Norway}
\email{andriybond@gmail.com}

\author[Titus Hilberdink]{Titus Hilberdink}
\address{Titus Hilberdink\\Department of Mathematics and Statistics\\University of Reading, Reading RG6 6AX, UK}
\email{t.w.hilberdink@reading.ac.uk}

\author[Kristian Seip]{Kristian Seip}
\address{Kristian Seip\\Department of Mathematical Sciences \\ Norwegian University of Science and Technology \\ NO-7491 Trondheim \\ Norway}
\email{seip@math.ntnu.no}

\thanks{Bondarenko and Seip were supported by Grant 227768 of the Research Council of Norway.}

\maketitle

\begin{abstract}
We prove that
\[
\sum_{k,{\ell}=1}^N\frac{(n_k,n_{\ell})^{2\alpha}}{(n_k
n_{\ell})^{\alpha}} \ll N^{2-2\alpha} (\log N)^{b(\alpha)}
\]
holds for arbitrary integers $1\le n_1<\cdots < n_N$ and $0<\alpha<1/2$ and show by an example that this bound is optimal, up to the precise value of the exponent $b(\alpha)$. This estimate complements recent results for $1/2\le \alpha \le 1$ and shows that there is no ``trace'' of the functional equation for the Riemann zeta function in estimates for such GCD sums when $0<\alpha<1/2$.    \end{abstract}

\section{Introduction}

The study of greatest common divisor (GCD) sums of the form
\begin{equation}\label{gcda}
\sum_{k,\ell=1}^N\frac{(n_k,n_{\ell})^{2\alpha}}{(n_k
n_{\ell})^\alpha}
\end{equation}
begins with G\'{a}l's theorem \cite{G} which asserts that when $\alpha=1$, $CN(\log \log N)^2$ is an optimal upper bound for \eqref{gcda}, with $C$ an absolute constant independent of $N$ and the distinct positive integers $n_1,...,n_N$ (the best possible value for $C$ is $6e^{2\gamma}/\pi^2$, where $\gamma$ is Euler's constant, as shown recently by Lewko and Radziwi{\l\l} \cite{LR}). Dyer and Harman \cite{DH},  motivated by applications in the metric theory of diophantine approximation, obtained the first estimates for the range $1/2 \le \alpha <1$. Recent work of Aistleitner, Berkes, and Seip \cite{ABS} for $1/2<\alpha<1$ and Bondarenko and Seip \cite{BS1, BS3} for $\alpha=1/2$ has led to the bounds
\begin{equation}\label{gcdab}
\sum_{k,\ell=1}^N\frac{(n_k,n_{\ell})^{2\alpha}}{(n_k
n_{\ell})^\alpha} \ll \begin{cases} N \exp\left(c(\alpha) \frac{(\log N)^{1-\alpha}}{(\log\log N)^{\alpha}} \right), & 1/2 < \alpha < 1 \\
  N \exp\left(A\sqrt{\frac{\log N \log\log\log N}{\log\log N}}\right), & \alpha=1/2, \end{cases}
\end{equation}
which are optimal, up to the precise values of the constants $c(\alpha)$ and $A$; the asymptotic behavior of $c(\alpha)$ has been clarified both when $\alpha \searrow 1/2 $ and $\alpha \nearrow 1$.

Bounds for the sum in \eqref{gcda} have a long history, and they have had a number of different applications; see the recent papers \cite{ABS, BS1, LR} and the references found there. In recent years, an additional interesting application has surfaced: Lower bounds for specific sums of the form \eqref{gcda} or corresponding quadratic forms have turned out to be useful for detecting large values of the Riemann zeta function $\zeta(s)$. This line of research was initiated in work of Soundararajan \cite{So} and Hilberdink \cite{H} and later pursued by Aistleitner \cite{A} who was the first to make the link to G\'{a}l-type estimates. Recently, using Soundararajan's resonance method \cite{So} and a certain large G\'{a}l-type sum for $\alpha=1/2$, Bondarenko and Seip showed that for every $c$, $0<c<1/\sqrt{2}$, there exists a $\beta$, $0<\beta<1$, such that the maximum of $|\zeta(1/2+it)|$ on the interval $T^{\beta}\le t \le T$ exceeds 
$\exp\left(c \sqrt{\log T \log\log\log T/\log\log T}\right)$ for all $T$ large enough. 

These developments have led us to look more closely at the ``phase transition'' at $\alpha=1/2$ by seeking estimates for \eqref{gcda} also in the range $0<\alpha<1/2$, which could possibly correspond to large values of $\zeta(\sigma+it)$ beyond the critical line $\sigma=1/2$. The present paper shows, however, that there is no symmetry in the estimates for \eqref{gcda}  when $\alpha$ is replaced by $1-\alpha$, as one might have expected from the functional equation for $\zeta(s)$. 

To state our main result, we let 
 $\mathcal{M}$ denote an arbitray  finite set of positive integers and introduce the quantity
\[ \Gamma_\alpha(N):=\max_{|\mathcal{M}|=N} \frac{1}{N} \sum_{m,n\in \mathcal{M}}\frac{{(m,n)}^{2\alpha}}{(mn)^{\alpha}}. \]
\begin{theorem}\label{main}
For every $\alpha$, $0<\alpha<1/2$, there exist positive
constants $a(\alpha)$ and $b(\alpha)$ such that
\begin{equation} \label{goal} N^{1-2 \alpha} (\log N)^{a(\alpha)} \le \Gamma_{\alpha}(N) \le N^{1-2 \alpha} (\log N)^{b(\alpha)} \end{equation}
for sufficiently large $N$.
\end{theorem}
 
Before giving the proof of this theorem, we will in the next section set the stage by considering the simpler but closely related question of finding the largest eigenvalue of the positive definite matrix $(m,n)^{2\alpha}/(mn)^{\alpha}$, $1\le m,n \le N$:
\begin{theorem}\label{spectral}
For every $\alpha$, $0<\alpha<1/2$, there exists a constant $C_{\alpha}$ such that \begin{equation}
\label{hh}
\max_{|a_1|^2+\cdots +|a_N|^2=1}\sum_{m,n\le N}\frac{a_m\bar{a}_n{(m,n)}^{2\alpha}}{(mn)^{\alpha}}\le C_\alpha N^{1-2\alpha}.
\end{equation}
\end{theorem}
We refer to \cite{H} for further information and for the precise asymptotics of the maximum in \eqref{hh} in the range $1/2\le \alpha \le 1$.

We notice that there is no logarithmic power in \eqref{hh}. Nevertheless, we will see that the idea for the proof of Theorem~\ref{spectral} (to be given in the next section) is used again as the starting point for the
proof of the bound from above in Theorem~\ref{main}. Resorting to some further ideas and estimates from 
\cite{BS1}, we will prove the latter bound in Section~\ref{abovesection}. In Section~\ref{example}, we construct an example  giving the inequality from below in \eqref{goal}. As expected, this example involves a large number of primes (a positive power of $N$). One may notice that it would not be possible to construct a similar example if we required the set to consist only of square-free numbers. Hence it remains an open problem to prove the analogue of Theorem~\ref{main} in the square-free case. More specifically, we may ask whether the logarithmic power can be discarded in this case as well.

\section{Proof of Theorem~\ref{spectral}}

We begin by noticing that
\[
S(a):=\sum_{d=1}^{N}\sum_{\substack{m,n\le N \\
{(m,n)}=d}}\frac{|a_ma_n|d^{2\alpha}}{(mn)^{\alpha}}
\le\sum_{d=1}^{N}\left(\sum_{m,n\le N/d}\frac{|a_{md}|}{m^{\alpha}}\right)^2.
\]
We introduce the multiplicative function $g(m):=\sum_{d|m}d^{-1/2+\alpha}$. By the Cauchy--Schwarz inequality,
\begin{equation}
\label{hh1}
S(a)\le\sum_{d=1}^{N}\left(\sum_{m\le N/d}\frac{|a_{md}|}{m^{\alpha}}\right)^2\le\sum_{d=1}^{N}\sum_{m\le N/d}\frac{|a_{md}|^2}{g(m)}
\sum_{m\le N/d}\frac{g(m)}{m^{2\alpha}}.
\end{equation}
To estimate the sum $\sum_{m\le N/d}\frac{g(m)}{m^{2\alpha}}$, we notice that
\[
 \sum_{n\le x} \frac{g(n)}{n^{2\alpha}}=\sum_{n\le x}\frac{1}{n^{2\alpha}}\sum_{d|n} \frac{1}{d^{\frac{1}{2}-\alpha}} =\sum_{d\le x} \frac{1}{d^{\frac{1}{2}+\alpha}} \sum_{n\le x/d}\frac{1}{n^{2\alpha}}\ll \sum_{d\le x} \frac{1}{d^{\frac{1}{2}+\alpha}}\Bigl(\frac{x}{d}\Bigr)^{1-2\alpha}\ll x^{1-2\alpha}.
\]
Hence by~\eqref{hh1} we have
\[
S(a)\ll N^{1-2\alpha}\sum_{d=1}^{N}\sum_{m\le N/d}\frac{|a_{md}|^2}{d^{1-2\alpha}g(m)}=N^{1-2\alpha}\sum_{n=1}^N|a_n|^2\sum_{d|n}\frac{d^{2\alpha-1}}{g(n/d)}.
\]
So to finish the proof of Theorem~\ref{spectral}, it is sufficient to show that 
\begin{equation}
\label{hh2}
h(n):=\sum_{d|n}\frac{d^{2\alpha-1}}{g(n/d)}
\end{equation}
is a bounded arithmetic function. We observe that $h(n)$ is a multiplicative function, which means that it suffices to consider
\[ h(p^m)=\sum_{\ell=0}^m \frac{p^{(2\alpha-1) \ell}}{g(p^{m-\ell})}\le \frac{1}{g(p^{m-r})}\sum_{\ell=0}^r p^{(2\alpha-1) \ell}+ \sum_{\ell=r+1}^m p^{(2\alpha-1) \ell} \]
for any $r$. Taking $r=[m/3]$ and using $1/g(p^m)\to 1-p^{\alpha-1/2}$ as $m\to\infty$ shows that $h(p^m)\le 1$ for $p^m$ sufficiently large. We infer from this  that $h(n)$ is a bounded arithmetic function.

As far as the numerical value of the constant $C_\alpha$ in Theorem~\ref{spectral} is concerned, we have confined ourselves to the following special case which seems to be of independent interest: 
\be \label{exact} \frac{1}{N}\sum_{m,n=1}^N \frac{(m,n)^{2\alpha}}{(mn)^{\alpha}}=\frac{\zeta(2-2\alpha)}{\zeta(2)(1-\alpha)^2} N^{1-2\alpha}+O(1). \ee 
\begin{proof}[Proof of \eqref{exact}] 
Write $F_{\alpha}(N)$ for the sum on the left and put $S_{\alpha}(x): = \sum_{\substack{m,n\le x \\ (m,n)=1}}  \frac{1}{(mn)^{\alpha}}$. Then
\[ F_{\alpha}(N) = \sum_{d\le N} \sum_{\substack{m,n\le N \\ (m,n)=d}}  \frac{(m,n)^{2\alpha}}{(mn)^{\alpha}} = \sum_{d\le N} S_{\alpha}(N/d).\]
Also let $T_{\alpha}(x) = \sum_{n\le x} \frac{1}{n^{\alpha}}= \frac{1}{1-\alpha}x^{1-\alpha}+O(1)$. Then
\[ T_{\alpha}(x)^2 = \sum_{m,n\le x}\frac{1}{(mn)^{\alpha}} = \sum_{d\le x}\frac{1}{d^{2\alpha}} \sum_{\substack{m,n\le x/d \\ (m,n)=1}}  \frac{1}{(mn)^{\alpha}} = \sum_{d\le x} \frac{1}{d^{2\alpha}}S_{\alpha}\Bigl(\frac{x}{d}\Bigr).\]
By M\"{o}bius inversion, $S_{\alpha}(x) =  \sum_{d\le x} \frac{\mu(d)}{d^{2\alpha}}T_{\alpha}(\frac{x}{d})^2$ and so
\[ F_{\alpha}(N) = \sum_{d\le N} \beta(n)T_{\alpha}\Bigl(\frac{N}{n}\Bigr)^2,\]
where $\beta(n) = \sum_{d|n}\frac{\mu(d)}{d^{2\alpha}}$. We note that $0<\beta(n)\le 1$ for all $n$. Thus
\[ F_{\alpha}(N) =  \frac{1}{(1-\alpha)^2} \sum_{n\le N} \beta(n)\Bigl( \Bigl(\frac{N}{n}\Bigr)^{2-2\alpha}+O\Bigl(\frac{N}{n}\Bigr)^{1-\alpha}\Bigr) =\frac{N^{2-2\alpha}}{(1-\alpha)^2} \sum_{n\le N} \frac{\beta(n)}{n^{2-2\alpha}} + O\biggl(N^{1-\alpha}\sum_{n\le N} \frac{1}{n^{1-\alpha}}\biggr).\]
The final term is $O(N)$, while $\sum_{n>N} \frac{\beta(n)}{n^{2-2\alpha}}\le\sum_{n>N} \frac{1}{n^{2-2\alpha}} \ll N^{2\alpha-1}$. Finally  
\[ \sum_{n=1}^\infty\frac{\beta(n)}{n^{2-2\alpha}} = \frac{\zeta(2-2\alpha)}{\zeta(2)}, \]
giving the result.
\end{proof}

\section{Proof of the bound from above in Theorem~\ref{main}}\label{abovesection}
In what follows, $\omega(n)$ denotes the number of distinct prime factors in $n$ and $d(n)$ is the divisor function.

We begin by stating the main auxiliary result used to prove the upper bound in \eqref{goal}. 

\begin{lemma}\label{divisor}
For every finite set $\mathcal{M}$ of positive integers there exists a divisor closed set  $\mathcal{M}'$ of positive integers with
$|\mathcal{M}'|=|\mathcal{M}|$ such that
\[ \sum_{m,n\in \mathcal{M}}\frac{{(m,n)}^{2\alpha}}{(mn)^{\alpha}}\le\sum_{m,n\in \mathcal{M}'}\frac{{(m,n)}^{2\alpha}}{(mn)^{\alpha}}
2^{\omega(mn/{(m,n)}^{2} )}. \]
\end{lemma}
\begin{proof} Following the proof of \cite[Lemma 2]{ABS}, we transform $\mathcal{M}$ into $\mathcal{M}'$ by means of the following algorithm. Fix a prime $p$ such that $p$ divides some number in $\mathcal{M}$. Then there exist distinct numbers $m_j$, $j=1,..., \ell$ with $\ell \le |\mathcal{M}|$ such that we may write
\[ \mathcal{M} =\bigcup_{j=1}^\ell \mathcal{M}_j,\]
where $\mathcal{M}_j$ consists of those $m$ in $\mathcal{M}$ such that $m/m_j$ is a power of $p$. We then replace the numbers in $\mathcal{M}_j$ by the numbers $m_j, m_j p,..., m_j p^{|\mathcal{M}_j|-1}$. This transformation is then performed for every prime dividing some number in $\mathcal{M}$. A close inspection of the largest possible change in the GCD sum in each step of this series of transformations (carried out in detail in the proof of \cite[Lemma~2]{ABS}) gives the desired estimate.  \end{proof}
We will also need the following two lemmas.
\begin{lemma}\label{square-free} 
%
%
Suppose that $0<\alpha<1/2$ and that $\beta$ is a real number. Then for every $\beta'>\beta/(2\alpha)$ there exists a positive constant $C$ with the following property. If $\mathcal{K}$ is a set of positive integers  with $|\mathcal{K}|=K$, then
\be \label{easy}
\sum_{m\in \mathcal{K}}\frac{d(m)^\beta}{m^{2\alpha}} \le C K^{1-2\alpha} [\log K]^{2^{\beta'}-1}.
\ee
\end{lemma}
\begin{proof} We begin by observing that
\[ \sum_{m\in \mathcal{K}}\frac{d(m)^\beta}{m^{2\alpha}}
\le \sum_{m=1}^{K} \frac{d(m)^\beta}{m^{2\alpha}} +
\sum_{\ell=0}^\infty \sum_{\substack{2^{\ell}K<m\le 2^{\ell+1} K \\ d(m)^\beta>2^{2\alpha \ell}}} \frac{d(m)^\beta}{m^{2\alpha}}. \]
Now
\begin{align*} \sum_{\substack{2^{\ell}K<m\le 2^{\ell+1} K \\ d(m)^\beta>2^{2\alpha \ell}}} \frac{d(m)^\beta}{m^{2\alpha}}
& \le 2^{-(\beta'/\beta-1)2\alpha \ell} \sum_{2^{\ell}K<m\le 2^{\ell+1} K} \frac{d(m)^{\beta'}}{m^{2\alpha}} \\
& \ll 2^{-(\beta'/\beta-1)2\alpha \ell}\cdot 2^{(1-2\alpha)\ell} K^{1-2\alpha} (\log 2^{\ell} K)^{2^{\beta'}-1}, \end{align*}
where we used the classical formula 
\[ \sum_{n\le x} d(n)^{\beta'} = B x (\log x)^{2^{\beta'}-1}\big(1+O((\log x)^{-1})\big) \] 
which holds with $B$ an absolute constant \cite{W, Se}. It follows that the sum over $\ell$ is dominated by a convergent geometric series if $\beta'>\beta/(2\alpha)$.
\end{proof}
We mention without proof that a more careful analysis shows that the exponent $2^{\beta'}-1$ on the right-hand side of \eqref{easy} can be replaced by 
$2\alpha(2^{\beta'}-1)$ with the same requirement that $\beta'>\beta/(2\alpha)$. Using results on the distribution of `large' values of $d(n)$ (see \cite{Nor}), we can show that this is optimal in the sense that the inequality fails with any exponent less than $2\alpha(2^{\beta/(2\alpha)}-1)$. 

\begin{lemma}\label{divideout} If $\mathcal{M}$ is a divisor closed set of square-free numbers, then  $|\frac{1}{p}\mathcal{M}|\le \frac{1}{2}|\mathcal{M}|$ for every prime $p$ in $\mathcal{M}$. 
\end{lemma}

\begin{proof} Suppose that $|\frac{1}{p}\mathcal{M}|=\ell$ and write $\frac{1}{p}\mathcal{M}=\{ m_1,\ldots, m_\ell \}$. Then $\mathcal{M}$ contains $pm_1,\ldots, p m_\ell$ and hence also $m_1,\ldots,m_M$, since it is divisor closed. As $\mathcal{M}$ is square-free, these numbers are all distinct, and so it follows that $|\mathcal{M}|\ge 2\ell$.
\end{proof}

We are now prepared to prove the bound from above in \eqref{goal}.
To begin with, we define
\[ \widetilde{\Gamma}_\alpha(N):=\max_{\mathcal{M} \text{ divisor closed}, |\mathcal{M}|=N} \frac{1}{N} \sum_{m,n\in \mathcal{M}}\frac{{(m,n)}^{2\alpha}}{(mn)^{\alpha}}
2^{\omega(mn/{(m,n)}^{2} )}. \]
By Lemma~\ref{divisor}, we have $\Gamma_\alpha(N)\le \widetilde{\Gamma}_\alpha(N)$, which means that it suffices to estimate
$\widetilde{\Gamma}_\alpha(N)$.
Hence we assume that the set $\mathcal{M}$ is divisor closed and estimate instead the sum
\[ \widetilde{S}:=\sum_{m,n\in \mathcal{M}}\frac{{(m,n)}^{2\alpha}}{(mn)^{\alpha}} 2^{\omega(mn/(m,n)^2)}.
\]
In what follows, $\mathcal{M}^*$ will denote the subset of $\mathcal{M}$ consisting of the square-free numbers in $\mathcal{M}$. In addition, given $m$ in $\mathcal{M}^*$, we let $\mathcal{M}(m)$ denote the subset of  $\mathcal{M}$ consisting of those numbers $n$ in $\mathcal{M}$ such that $p|n$ if and only if $p|m$. Hence
\[ \mathcal{M}=\bigcup_{m\in \mathcal{M}^*} \mathcal{M}(m) \quad \text{and} \quad \sum_{m\in \mathcal{M}^*} |\mathcal{M}(m)|=N. \]

Now suppose that $k$ and $\ell$ are in $\mathcal{M}^*$ and that $|\mathcal{M}(k)|\ge |\mathcal{M}(\ell)|$. We then find that
\begin{align*} \sum_{m\in \mathcal{M}(k), n\in \mathcal{M}(\ell)}\frac{{(m,n)}^{2\alpha}}{(mn)^{\alpha}} 2^{\omega(mn/(m,n)^2)}
& \le \frac{{(k,\ell)}^{2\alpha}}{(k\ell)^{\alpha}} 2^{\omega(k\ell/(k,\ell)^2)}\sum_{n \in \mathcal{M}(\ell)} \prod_{p|k}\big(1+4\sum_{\nu=1}^\infty p^{-\nu} \big) \\
& \ll \frac{{(k,\ell)}^{2\alpha}}{(k\ell)^{\alpha}} 2^{\omega(k\ell/(k,\ell)^2)}|\mathcal{M}(\ell)| d(k)^{\varepsilon},
\end{align*}
where the implicit constant in the latter relation only depends on $\alpha$ and $\varepsilon$. Here $\varepsilon$ can be any positive number, but in what follows we will require that $0<\varepsilon < 1-2\alpha$. We infer from the latter relation that
\[ \widetilde{S}\ll \sum_{m,n\in \mathcal{M}^*} |\mathcal{M}(m)|^{1/2} d(m)^{\varepsilon} |\mathcal{M}(n)|^{1/2} d(n)^{\varepsilon} \frac{{(m,n)}^{2\alpha}}{(mn)^{\alpha}} 2^{\omega(mn/(m,n)^2)}. \]
This leads to the bound
\[
\widetilde{S}\ll \sum_{k\in \mathcal{M}^*} \left(\sum_{m\in \frac{1}{k} \mathcal{M}^*} \frac{|\mathcal{M}(mk)|^{1/2} d(mk)^{\varepsilon} d(m)^{1+\varepsilon}}{m^{\alpha}} \right)^2.
\]
By the Cauchy--Schwarz inequality, we obtain from this that
\[ \widetilde{S}\ll \sum_{k\in \mathcal{M}^*} d(k)^{2\varepsilon} \sum_{n\in \frac{1}{k} \mathcal{M}^*}\frac{|\mathcal{M}(nk)|}{d(n)^{\beta}} \sum_{m\in \frac{1}{k} \mathcal{M}^*} \frac{d(m)^{\beta+2+4\varepsilon}}{m^{2\alpha}},\]
where $\beta$ is a positive parameter to be chosen later.
Using Lemma~\ref{square-free} and the estimate
\[ |\frac{1}{k}\mathcal{M}|\le N 2^{-\omega(k)},\]
which we get from Lemma~\ref{divideout}, we therefore get
\begin{align*}
\widetilde{S} & \ll N^{1-2\alpha}(\log N)^{2^{\beta'}-1} \sum_{k\in \mathcal{M}^*} d(k)^{2\alpha+2\varepsilon-1} \sum_{n\in \frac{1}{k} \mathcal{M}^*}\frac{|\mathcal{M}(nk)|}{d(n)^{\beta}} \\
& =N^{1-2\alpha}(\log N)^{2^{\beta'}-1}\sum_{m\in \mathcal{M}^*}|\mathcal{M}(m)| \sum_{k|m}\frac{1}{2^{(1-2(\alpha+\varepsilon))\omega(k)+\beta\omega(n/k)}}
\end{align*}
with $\beta'> (\beta+2+4\varepsilon)/(2\alpha)$. Since
\[ n\mapsto \sum_{d|n}\frac{1}{2^{(1-2(\alpha+\varepsilon))\omega(d)+\beta\omega(n/d)}} \]
is a multiplicative function, and $n$ is squarefree, it suffices to make sure that
\[ \frac{1}{2^{1-2(\alpha+\varepsilon)}}+\frac{1}{2^{\beta}}\le 1. \]
This means that we need
\[ \beta\ge \frac{\log\frac{1}{1-2^{2(\alpha+\varepsilon) -1}}}{\log 2} \] to obtain the uniform bound
\[ \sum_{k|m}\frac{1}{2^{(1-2(\alpha+\varepsilon))\omega(k)+\beta\omega(n/k)}}\le 1. \]
We then find  that
\[ \widetilde{S} \ll N^{1-2\alpha}(\log N)^{2^{\beta'}-1}\sum_{m\in \mathcal{M}^*} |\mathcal{M}(m)|=N^{2-2\alpha}(\log N)^{2^{\beta'}-1}, \]
which in turn leads to the desired conclusion.

\section{Proof of the bound from below in Theorem~\ref{main}}\label{example}
This section will make extensive use of the Euler totient function $\phi(n)$. We will also need an additional multiplicative function, namely
\[ f(n):=\prod_{p|n}\frac{\left(p^{2\alpha-1}\left(1-\frac 1p\right)^{2\alpha}+\left(1-\frac 1p\right)\right)}{\left(\frac 1p\left(1-\frac 1p\right)+\left(1-\frac 1p\right)^{2-2\alpha}\right)}. \]

We now fix a positive number $M$ and set $k:=\prod_{p\le M}p$. We will need the following lemma.
\begin{lemma}
\label{mult}
For every $c$ such that $c|k$, we have
\be
\label{mult1}
\frac 1{k^{2}}\sum_{d|k}(c,d)^{2\alpha}\Big(\phi\big(\frac{k}{c}\big)\phi\big(\frac{k}{d}\big)\Big)^{\alpha}\big(\phi(c)\phi(d)\big)^{1-\alpha}
=\prod_{p|k}\left(\frac 1p\left(1-\frac 1p\right)+\left(1-\frac 1p\right)^{2-2\alpha}\right)\frac ckf\big(\frac kc\big).
\ee
Moreover, we have
%
%
\be
\label{mult10}
\frac 1{k^{2}}\!\sum_{c,d|k}\!(c,d)^{2\alpha}\Big(\phi\big(\frac{k}{c}\big)\phi\big(\frac{k}{d}\big)\Big)^{\alpha}\big(\phi(c)\phi(d)\big)^{1-\alpha}
\!=\!\prod_{p|k}\!\left(p^{2\alpha\!-\!2}\left(1-\!\frac 1p\right)^{2\alpha}\!+\!\frac 2p\left(1-\!\frac 1p\right)\!+\!\left(1-\!\frac 1p\right)^{2\!-\!2\alpha}\right).
\ee
\end{lemma}
\begin{proof}
Since every divisor $d$ of $k$ has a unique representation $d=d_1d_2$, where $d_1|c$ and $d_2|\frac kc$ we find that
the left-hand side $LHS$ of \eqref{mult1} is 
\begin{align}
LHS= & \frac{1}{k^2}\phi(k/c)^{\alpha}\phi(c)^{1-\alpha}\sum_{d_1|c}\sum_{d_2|\frac kc}d_1^{2\alpha}\phi\big(\frac c{d_1}\big)^{\alpha}\phi\big(\frac{k/c}{d_2}\big)^\alpha
\phi(d_1)^{1-\alpha}\phi(d_2)^{1-\alpha} \nonumber
\\
= & \frac{1}{k^2}\phi(k/c)^{\alpha}\phi(c)^{1-\alpha}\left(\sum_{d_1|c}d_1^{2\alpha}\phi\big(\frac c{d_1}\big)^{\alpha}\phi(d_1)^{1-\alpha}\right)
\left(\sum_{d_2|\frac kc}\phi\big(\frac{k/c}{d_2}\big)^\alpha\phi(d_2)^{1-\alpha}\right). \label{return}
\end{align}
Using the formula $\phi(d)=d\prod_{p|d}\left(1-\frac 1p\right)$ and the fact that the respective sums represent multiplicative functions, we find that
\[
\sum_{d_1|c}d_1^{2\alpha}\phi\big(\frac c{d_1}\big)^{\alpha}\phi(d_1)^{1-\alpha}=\prod_{p|c}\left(p^\alpha\left(1-\frac 1p\right)^\alpha+p^{1+\alpha}\left(1-\frac 1p\right)^{1-\alpha}\right),
\]
and
\[
\sum_{d_2|\frac kc}\phi\big(\frac{k/c}{d_2}\big)^\alpha\phi(d_2)^{1-\alpha}=\prod_{p|\frac kc}\left(p^\alpha\left(1-\frac 1p\right)^\alpha+p^{1-\alpha}\left(1-\frac 1p\right)^{1-\alpha}\right).
\]
Returning to \eqref{return} and using again that $\phi(d)=d\prod_{p|d}\left(1-\frac 1p\right)$, we therefore obtain
\begin{align*}
LHS = & \frac 1{k^{2}}\sum_{d|k}(c,d)^{2\alpha}\Big(\phi\big(\frac{k}{c}\big)\phi\big(\frac{k}{d}\big)\Big)^{\alpha}\big(\phi(c)\phi(d)\big)^{1-\alpha}
\\
= & \frac ck\prod_{p|c}\left(\frac 1p\left(1-\frac 1p\right)+\left(1-\frac 1p\right)^{2-2\alpha}\right)
\prod_{p|\frac kc}\left(p^{2\alpha-1}\left(1-\frac 1p\right)^{2\alpha}+\left(1-\frac 1p\right)\right)
\\
= & \prod_{p|k}\left(\frac 1p\left(1-\frac 1p\right)+\left(1-\frac 1p\right)^{2-2\alpha}\right)\frac ckf\big(\frac kc\big),
\end{align*}
where the last expression is the right-hand side of \eqref{mult1}. Finally, we get~\eqref{mult10} from~\eqref{mult1} by using that $n\mapsto \sum_{d|n} \frac{1}{d}f(d)$ is a multiplicative function.
\end{proof}
In addition to the identities of the preceding lemma, we need following quantitative estimate.
\begin{lemma}
\label{prod}
For every $\alpha$, $0<\alpha<1/2$, there exists a positive constant $c_\alpha$ such that
$$
\prod_{p\le M}\left(p^{2\alpha-2}\left(1-\frac 1p\right)^{2\alpha}+\frac 2p\left(1-\frac 1p\right)+\left(1-\frac 1p\right)^{2-2\alpha}\right)\ge
c_\alpha(\log M)^{2\alpha}.
$$
\end{lemma}
\begin{proof}
The result follows from the fact that
$$
p^{2\alpha-2}\left(1-\frac 1p\right)^{2\alpha}+\frac 2p\left(1-\frac 1p\right)+\left(1-\frac 1p\right)^{2-2\alpha}=1+\frac{2\alpha}p+O(p^{2\alpha-2}),
\quad{p\to\infty},
$$
along with Mertens's third theorem, i.e., the fact that $\prod_{p\le M} (1-1/p)\sim e^{\gamma}/\log M$ when \mbox{$M\to \infty$}.
\end{proof}
The following theorem yields the bound from below in Theorem~\ref{main}. 
\begin{theorem}
\label{mult2}
For every $\alpha$, $0<\alpha<1/2$, there exists a positive constant $c_\alpha$ such that if $N$ is a positive integer, then there exists a set of integers $\mathcal{M}$ of cardinality $N$ such that
\[
\sum_{m,n\in \mathcal{M}}\frac{{(m,n)}^{2\alpha}}{(mn)^{\alpha}}\ge c_\alpha N^{2-2\alpha}(\log N)^{2\alpha}.
\]
\end{theorem}
\begin{proof}
Fix $\alpha$, $0<\alpha<1/2$, and let $N$ be positive integer. Set $M=N^{\delta}$, where $\delta$, $0<\delta<1$, is a constant depending only on $\alpha$ to be chosen later,
$k=\prod_{p\le M}p$. Let $\mathcal{A}$ be the set of the first $[N^{1/3}]$ $M$-smooth square-free numbers and $\mathcal{D}$ be the set of integers of the form $k/a$ with $a$ in $\mathcal{A}$. For every number $d$ in $\mathcal{D}$ denote by $S_d$ the set of the first $[N\phi(d)/k]$ integers $s$ such that $(s,k/d)=1$, and by $s_d$ the maximal number in $S_d$. Also, let $dS_d$ be the set of integers of the form $ds$, where $s\in S_d$ and $d\in\mathcal{D}$. Finally, set
$\mathcal{M}:=\bigcup_{d\in\mathcal{D}}dS_d$.

It is clear that all numbers $d$ in $\mathcal{D}$ are square-free and also that the sets $dS_d$ are pairwise disjoint.
Moreover, since $\sum_{d|k}\phi(d)=k$ we have that $|\mathcal{M}|<N$. Also, 
\be \label{sd} |S_d|\ge \frac{N^{2/3}}{2 \log N} \ee
for sufficiently large $N$ and every $d$ in $\mathcal{D}$; this follows from the formula $\phi(d)=d\prod_{p|d}\left(1-\frac{1}{p}\right)$ and an application Mertens's third theorem.  
We can use this to get an upper bound for $s_d$. Indeed,  each set $S_d$
is just a set of numbers of the form $a\mod (k/d)$, where $a$ is from a set of cardinality $\phi(k/d)$.
Therefore if $d$ is in $\mathcal{D}$, then
$$
s_d\le \frac{2N\phi(d)}{d\phi(k/d)}.
$$
Here we used that $|S_d|\ge k/d$ which follows from \eqref{sd}.
Now we have, using \eqref{sd} in the last step, 
\begin{align*}
\sum_{m,n\in \mathcal{M}}\frac{{(m,n)}^{2\alpha}}{(mn)^{\alpha}} & \ge \sum_{c,d\in \mathcal{D}}\frac{{(c,d)}^{2\alpha}}{(cd)^{\alpha}}
\sum_{m\in S_c}\frac{1}{m^\alpha}\sum_{n\in S_d}\frac{1}{n^\alpha}
\ge\sum_{c,d\in \mathcal{D}}\frac{{(c,d)}^{2\alpha}}{(cd)^{\alpha}}
\frac{|S_c||S_d|}{s_c^\alpha s_d^\alpha} \\
& \ge c_\alpha N^{2-2\alpha}\frac 1{k^{2}}\sum_{c,d\in \mathcal{D}}(c,d)^{2\alpha}(\phi(k/c)\phi(k/d))^{\alpha}(\phi(c)\phi(d))^{1-\alpha}
\end{align*}
for some positive constant $c_\alpha$ depending only on $\alpha$. In view of~\eqref{mult10} of Lemma~\ref{mult} and Lemma~\ref{prod},  the proof will be complete if we can prove that 
$$
\sum_{c,d\in \mathcal{D}}(c,d)^{2\alpha}(\phi(k/c)\phi(k/d))^{\alpha}(\phi(c)\phi(d))^{1-\alpha}\ge 
\frac 13 \sum_{c,d|k}(c,d)^{2\alpha}(\phi(k/c)\phi(k/d))^{\alpha}(\phi(c)\phi(d))^{1-\alpha}.
$$
The latter inequality will follow from the bound
\[
\sum_{c\not\in\mathcal{D},d|k}(c,d)^{2\alpha}(\phi(k/c)\phi(k/d))^{\alpha}(\phi(c)\phi(d))^{1-\alpha}\le 
\frac 13 \sum_{c,d|k}(c,d)^{2\alpha}(\phi(k/c)\phi(k/d))^{\alpha}(\phi(c)\phi(d))^{1-\alpha}.
\]
By~\eqref{mult1}, this is equivalent to
\be
\label{final}
\sum_{n\in \mathcal{F}}\frac{f(n)}{n}\le \frac 13\prod_{p\le M}\left(1+\frac{f(p)}{p}\right),
\ee
where $\mathcal{F}$ is the set of all $M$-smooth square-free numbers larger than $[N^{1/3}]$.
By Rankin's trick, we have
\begin{align*}
\sum_{n\in \mathcal{F}}\frac{f(n)}{n} & \le N^{-\frac{1}{3\delta\log N}}\prod_{p\le M}\left(1+\frac{f(p)}pp^{\frac{1}{\delta\log N}}\right) \\
& \le e^{-\frac{1}{3\delta}}\prod_{p\le M}\left(1+\frac{f(p)}{p}\right)\prod_{p\le M}\left(1+\frac{f(p)}{p}\left(p^{\frac{1}{\delta\log N}}-1\right)\right).
\end{align*}
Now we note that the second product is bounded by a constant depending only on $\alpha$ (not on $\delta$).
Indeed,
$$
\sum_{p\le M}\frac{f(p)}{p}\left(p^{\frac{1}{\delta\log N}}-1\right)\le 2\sum_{p\le M}\frac{f(p)-1}{p}+2\sum_{p\le M}\frac{\log p}{p\log M}.
$$ 
The first sum is bounded because $f(p)=1+p^{2\alpha-1}+o(p^{2\alpha-1})$ as $p\to\infty$, and the second sum is bounded since $\sum_{p\le M} (\log p)/p- \log M\le 2$ by Mertens's first theorem.
Therefore, choosing $\delta$ sufficiently small (depending only on $\alpha$), we get~\eqref{final}. Theorem~\ref{mult2} is proved.
\end{proof}

\end{document}